\documentclass[leqno,12pt]{amsart}
\usepackage{amsmath}
\usepackage{amssymb,amscd}
\usepackage{amsthm}
\usepackage{latexsym}
\usepackage[myheadings]{fullpage}
\usepackage{color}  

\newtheorem{thm}{Theorem}[section]
\newtheorem{lem}[thm]{Lemma}
\newtheorem{prop}[thm]{Proposition}
\newtheorem{cor}[thm]{Corollary}
\newtheorem{rem}[thm]{Remark}

\newtheorem{ex}[thm]{Example}
\newtheorem{question}[thm]{Question}

\newcommand{\RI}{R \!\Join\! I}
\newcommand{\R}{R[It]}
\newcommand{\m}{\mathfrak m}
\newcommand{\p}{\mathfrak p}
\newcommand{\q}{\mathfrak q}
\newcommand{\Z}{{\mathbb Z}}

\newcommand{\Hom}{{\rm Hom}}
\newcommand{\Ann}{{\rm Ann}}
\newcommand{\RR}{R(I)_{a,b}}

\def\opn#1#2{\def#1{\operatorname{#2}}} 
\opn\spec{Spec}
\opn\depth{depth}
\opn\height{ht}

\title{New algebraic properties of quadratic quotients of the Rees algebra}

\author{Marco D'Anna}
\address{Marco D'Anna - Dipartimento di Matematica e Informatica - Universit\`a degli Studi di Catania - Viale Andrea Doria 6 - 95125 Catania - Italy}\email{mdanna@dmi.unict.it}

\author{Francesco Strazzanti}
\address{Francesco Strazzanti - Departamento de \'Algebra \&  Instituto de Matem\'aticas (IMUS) -
Universidad de Sevilla - Avda. Reina Mercedes s/n -
41012 Sevilla - Spain}
\email{francesco.strazzanti@gmail.com}

\thanks{The second author was partially supported by MTM2016-75027-P, MTM2013-46231-P (Ministerio de Economı\'ia y Competitividad) and FEDER}

\subjclass[2010]{13B30, 13H10}

\date{\today}

\begin{document}

\begin{abstract}
We study some properties of a family of rings $R(I)_{a,b}$ that are obtained as quotients of the Rees algebra associated with a ring $R$ and an ideal $I$. In particular, we give a complete description of the spectrum of every member of the family and describe the localizations at a prime ideal. Consequently, we are able to characterize the Cohen-Macaulay and Gorenstein properties, generalizing known results stated in the local case. Moreover, we study when $R(I)_{a,b}$ is an integral domain, reduced, quasi-Gorenstein, or satisfies Serre's conditions.
\end{abstract}

\keywords{Idealization; Amalgamated duplication; Quasi-Gorenstein rings; localizations; Serre's conditions}

\maketitle

\section*{Introduction}

Let $R$ be a commutative ring with unity, let $I \neq 0$ be a proper ideal of $R$ and let $a,b \in R$. In \cite{BDS} the authors introduce and study the family of quotient rings 
$$R(I)_{a,b}=\frac{\R}{(I^2(t^2+at+b))} \ ,$$
where $\R=\bigoplus_{n \geq 0}(I^nt^n)$ is the Rees algebra associated with the ring $R$ with respect to $I$ and
$(I^2(t^2+at+b))$ is the contraction to $\R$ of the ideal
generated by $t^2+at+b$ in $R[t]$.

This family provides a unified approach to Nagata's idealization (with respect to an ideal, see \cite[pag. 2]{n}) and to amalgamated duplication (see \cite{D'A} and \cite{DF}); they can be both obtained as particular members of the family, more precisely they are isomorphic to $R(I)_{0,0}$ and $R(I)_{0,-1}$ respectively. This fact explains why these
constructions produce rings with many common properties; as a
matter of fact, it is shown, in \cite{BDS}, that many properties
of the rings in this family (like, e.g., Krull dimension,
noetherianity and local Cohen-Macaulayness) do not depend on the defining polynomial. One interesting fact about this family is that, if $R$ is an integral domain, we can always find integral domains among its members, whereas the idealization is never reduced and the amalgamated duplication is never an integral domain. Hence, this construction revealed to be useful to construct $R$-algebras (and, in particular, integral domains)
satisfying pre-assigned properties. For instance, in \cite{BDS} it is given a formula for the Hilbert function of $R(I)_{a,b}$ in the local case, that depends only on $R$ and $I$, and in \cite{OST} this formula is used to construct families of one-dimensional Gorenstein local rings with decreasing Hilbert function, solving a problem formulated by M.E. Rossi. 

In a subsequent paper (cf. \cite{BDS2}), the same authors deepen the study of this family of rings in the local case, characterizing when its members are Gorenstein, complete intersection, and almost Gorenstein, proving that these properties do not depend on the particular member chosen in the family, but only on $R$ and $I$. We also notice that other properties of these quotients have been studied in \cite{F}.

In this paper we are interested in understanding the prime spectrum of the rings in the family, in relation to the prime spectrum of the original ring, and in studying the behaviour of the localizations. 
More precisely, we explicitly describe the primes of $R(I)_{a,b}$ lying over a fixed prime $\mathfrak p$ of $R$,
showing that two cases can arise, depending on the reducibility of $t^2+at+b$ in $Q(R/\mathfrak p)[t]$
and (if the polynomial is reducible) on its roots. In case there is only one prime $\mathfrak q$ lying over $\mathfrak p$, then we obtain $(\RR)_{\mathfrak q} \cong R_{\p}(I_{\p})_{a,b}$, while, if there exist two primes
$\mathfrak p_1$ and $\mathfrak p_2$ lying over $\mathfrak p$, then $(\RR)_{\p_i} \cong R_\p$ for $i=1,2$ provided that a technical hypothesis holds (see Proposition \ref{localizations}).

This facts will allow to extend in a natural way some local results contained in the papers cited above and to investigate other relevant properties that depend on localizations, like Serre's conditions. We also notice that the study of localizations has an intrinsic interest and can be applied in many situations
like, e.g., the geometric contest: if we start with a finitely generated $k$-algebra $R$, all the rings in the family are finitely generated $k$-algebras and we can determine whether a prime ramifies or not.

Finally, under the hypotheses that  $t^2+at+b$ is reducible in $R[t]$ and $I$ is regular, we prove that $R(I)_{a,b}$ is quasi-Gorenstein if and only if $R$ satisfies the Serre's condition $(S_2)$ and $I$ is a canonical ideal of $R$. The concept of quasi-Gorenstein rings arises from the theory of linkage and it is a generalization of the notion of Gorenstein rings in the non-Cohen-Macaulay case. It is already known when idealization and amalgamated duplication are quasi-Gorenstein (see \cite{A2} and \cite{BSTY}) and we extend these results to the more general case $R(I)_{a,b}$.
\medskip

The structure of the paper is the following. In the first section we give a complete description of the prime spectrum of $R(I)_{a,b}$ (see Proposition \ref{minimal primes}) and describe the localizations of $R(I)_{a,b}$ at prime ideals (see Proposition \ref{localizations}); as a corollary, we characterize when $R(I)_{a,b}$ is an integral domain and when it is a Cohen-Macaulay or a Gorenstein ring. 
In the second section we study Serre's conditions for $R(I)_{a,b}$ (see Propositions \ref{Sn} and \ref{Rn}). 

Finally, in the last section we consider the particular case $t^2+at+b$ reducible in $R[t]$. More precisely, we study when $R(I)_{a,b}$ is isomorphic either to the idealization or to the amalgamated duplication and we characterize the properties of being reduced and quasi-Gorenstein (see Proposition \ref{reduced} and Theorem \ref{quasi Gorenstein}).

\section{Spectra and localizations}

Throughout this paper $R$ is a commutative ring with identity and $I \neq (0)$ a proper ideal of $R$;
with $Q(R)$ we will denote its total ring of fractions.
In this section we study the prime spectrum of the rings $\RR$. To this aim we recall that the extensions $R \subset \RR \subset R[t]/(t^2+at+b)$ are integral and, given $\p \in \spec R$, the prime ideals of $\RR$ lying over $\p$ depend on the reducibility of $t^2+at+b$ on $Q(R/\p)$, the field of fractions of $R/\p$, see \cite[Proposition 1.9]{BDS}. More precisely, if $t^2+at+b$ is irreducible in $Q(R/\p)[t]$, there will be only one prime of $R[t]/(t^2+at+b)$ lying over $\p$ and, thus, the same holds for $\RR$; on the other hand, when $t^2+at+b$ is reducible, there are two primes of $R[t]/(t^2+at+b)$ lying over $\p$ and they can either contract to the same prime in $\RR$ or to two different primes, depending on $I$, $\p$, and on the factorization of the polynomial.

Assuming that $t^2+at+b$ is reducible in $Q(R/\p)[t]$ and that $\bar\alpha / \bar\gamma$ and $\bar\beta / \bar\delta$ are its roots, we can always choose a representation with $\bar\gamma=\bar\delta$. In this case, it is easy to see that in $Q(R/\p)$ we have $\bar\gamma \bar a=-\bar\alpha-\bar\beta$ and $\bar\gamma^2 \bar b=\bar\alpha \bar\beta$ and, clearly, the same equalities hold in $R/\p$.
We start with a preparatory lemma.

\begin{lem} \label{ideals}
Let $\p$ be a prime ideal of $R$ and suppose that
$t^2+at+b=(t-\bar\alpha/ \bar\gamma)(t- \bar\beta/\bar\gamma)$ in $Q(R/\p) [t]$. Let $\alpha, \beta, \gamma \in R$ such that their classes modulo $\p$ are, respectively, $\bar\alpha, \bar\beta$ and $\bar\gamma$. Then, the two sets
$$
\p_1:=\{r+it \ | \ r \in R, i \in I, \gamma r + \alpha i \in \p \}
$$
$$
\p_2:=\{r+it \ | \ r \in R, i \in I, \gamma r + \beta i \in \p \}
$$
do not depend on the choice of $\alpha$, $\beta$, and $\gamma$ and are prime ideals of $\RR$. Moreover, $\p_1=\p_2$
if and only if $(\alpha-\beta)I \subseteq \p$.
\end{lem}

\begin{proof}
The fact that the sets defined above do not depend on the choice of $\alpha$, $\beta$ and $\gamma$ is an easy verification.
In order to prove that they are prime ideals, it is enough to consider $\p_1$, since the proof for $\p_2$ is analogous.

By definition $\gamma \notin \p$, then, $\gamma r + \alpha i\in \p$ if and only if $\gamma(\gamma r + \alpha i)\in \p$. Let $r+it \in \p_1$ and $s+jt \in \RR$. We have $(r+it) (s+jt)=rs-ijb+(rj+si-aij)t$ and, since
$\gamma a +\alpha + \beta$, $\gamma^2 b - \alpha \beta$, $\gamma^2 r + \gamma\alpha i\in \p$, we get, modulo $\p$,
\begin{gather*}
\gamma(\gamma(rs-ijb)+ \alpha (rj+si-aij)) = \\
= \gamma^2 rs - ij \gamma^2 b+ \gamma \alpha rj + \gamma \alpha si- \gamma \alpha aij \equiv \\
\equiv \gamma^2 rs -ij \alpha \beta + \gamma \alpha rj -\gamma^2 rs + \alpha^2 ij+\alpha\beta ij = \\
= \gamma \alpha rj +\alpha^2 ij = \alpha j (\gamma r + \alpha i) \equiv 0 
\end{gather*}
and this means that $(r+it)(s+jt)$ is in $\p_1$, i.e. $\p_1$ is an ideal.

Now we have to prove that $\p_1$ is prime. If we suppose that $(r+it)(s+jt) \in \p_1$, then $\gamma rs - \gamma ijb  + \alpha rj +\alpha si - \alpha ija \in \p$ and, multiplying by $\gamma$, we get, modulo $\p$,
\begin{gather*}
\gamma^2 rs - ij \alpha \beta + \gamma \alpha rj+\gamma \alpha si  + \alpha^2 ij+ij \alpha \beta\equiv \\
\equiv \gamma s (\gamma r+\alpha i)+ \alpha j (\gamma r+ \alpha i)\equiv 
(\gamma s + \alpha j) (\gamma r + \alpha i)\in \p.
\end{gather*}
Since $\p$ is prime, at least one between $\gamma s + \alpha j$ and $\gamma r + \alpha i$ belongs to $\p$, i.e. at least one between $r+it$ and $s+jt$ is in $\p_1$.

As for the last statement, suppose first that $(\alpha-\beta)I \subseteq \p$. If $r+it \in \p_1$, then $\gamma r+ \beta i = \gamma r + \alpha i - (\alpha - \beta) i \in \p$; therefore $\p_1 \subseteq \p_2$ and the other inclusion is analogous. Conversely, we first notice that for every $i \in I$, $- \alpha i + \gamma it \in \p_1$. Since $\p_1=\p_2$, we get $- \gamma \alpha i + \beta i \gamma \in \p$ and, therefore, $-\gamma (\alpha - \beta)i \in \p$. Hence, $(\alpha-\beta)i \in \p$, since $\gamma \notin \p$ and $\p$ is a prime ideal.
\end{proof}

\begin{prop} \label{minimal primes}
Let $\p$ be a prime ideal of $R$.
\begin{enumerate}
\item If $t^2+at+b$ is irreducible in $Q(R/\p)[t]$, then the only prime ideal of $\RR$ lying over $\p$ is
$\q:= \{p+it \ | \ p \in \p, i \in I \cap \p \}$.
\item If $t^2+at+b=(t-\bar \alpha / \bar\gamma)(t- \bar\beta / \bar\gamma)$ in $Q(R/\p)[t]$, then the ideals $\p_1$ and $\p_2$ defined in the previous lemma are the only prime ideals of $\RR$ lying over $\p$.
\end{enumerate}
\end{prop}

\begin{proof}
The first case is straightforward, because the prime ideal of $\RR$ lying over $\p$ is
$$\p \frac{R[t]}{(t^2+at+b)} \cap \RR = \{p+it \, | \, p \in \p, i \in I \cap \p \}.$$
As for the second case, we easily observe that $\p_1 \cap R = \p = \p_2 \cap R$ and, hence, $\p_1$ and $\p_2$ are prime ideals lying over $\p$.
In fact, in the proof of \cite[Proposition 1.9]{BDS} it is proved that a prime ideal of $\RR$ lying over $\p$ is the contraction to $\RR$ of the ideals $J=\varphi_{\p}^{-1}((t-\bar\alpha / \bar\gamma))$ and $H=\varphi_{\p}^{-1}((t-\bar\beta / \bar\gamma))$, where $\varphi_{\p}$ is the composition of the canonical homomorphisms $R[t] \rightarrow (R/\p)[t]
\hookrightarrow Q(R/\p)[t]$. 
Since $J$ and $H$ contain $\p R[t]$, it is easy to see that the extensions of $\p_1$ and $\p_2$ in $R[t]/(t^2+at+b)$ are contained in the images $\bar J$ and $\bar H$ in the same ring of $J$ and $H$, respectively. In fact, if $r+it \in \p_1$, then $\gamma r + \alpha i \in \p$ and
$$
\varphi_{\p}(r+it)=\frac{\bar r \bar\gamma}{\bar\gamma} +\bar i t = \bar i t -\frac{\bar\alpha \bar i}{\bar \gamma} + \frac{\bar\gamma \bar r + \bar\alpha \bar i}{\bar\gamma}= 
\bar i\left(t-\frac{\bar\alpha}{\bar\gamma}\right)\in \left(t-\frac{\bar\alpha}{\bar\gamma}\right)Q\left(\frac{R}{\p}\right)[t];
$$
therefore, $r+it \in \bar J\cap\RR$. Analogously, $\p_2 \subset \bar H\cap \RR$. By Incomparability, see \cite[Corollary 4.18]{E}, we get that $\p_1=\bar J\cap \RR$ and $\p_2=\bar H\cap \RR$ and this concludes the proof.
\end{proof}

In \cite[Remark 1.10.3]{BDS} it is noted that $\R(I)_{a,b}$ is an integral domain, if $t^2+at+b$ is irreducible in $Q(R)[t]$ and $R$ is an integral domain; thanks to the previous proposition we can prove the converse.

\begin{cor} \label{domain}
$\RR$ is an integral domain if and only if $R$ is an integral domain and $t^2+at+b$ is irreducible in $Q(R)[t]$.
\end{cor}

\begin{proof}
Since $R \subseteq \RR$, we can assume that $R$ is an integral domain.
If $t^2+at+b$ is irreducible in $Q(R)[t]$, the ideal
$\mathfrak q= \{p+it \ | \ p \in (0)R, i \in I \cap (0)R \}=(0)\RR$ is prime and, thus, $\RR$ is an integral domain.
Conversely, suppose by contradiction that $t^2+at+b$ is reducible in $Q(R)[t]$ and let $\p_1, \p_2$ be the prime ideals of $\RR$ lying over $(0)R$.
They are minimal primes of $\RR$ and, since it is an integral domain, they are equal to $(0)\RR$. On the other hand, it is easy to see that, for any $i \in I$, the element $i \alpha - i \gamma t$ is in $\p_1$ and it is different from zero, because $R$ is an integral domain; contradiction.
\end{proof}

In order to study the behaviour of localizations, notice that, since the ideals $\p_i$ are independent of the choice of the elements $\alpha$, $\beta$ and $\gamma$, we can choose them in such a way that $a\gamma=-\alpha-\beta$ and $b\gamma^2=\alpha \beta+p$ in $R$, where $p \in \p$. 
This choice is not unique, in fact for any $q \in \p$, substituting $\alpha$ with $\alpha+q$ and $\beta$
with $\beta-q$, the first equality still holds and the second one is modified up to a summand in $\p$.

\begin{prop} \label{localizations}
Let $\p$ be a prime ideal of $R$.
\begin{enumerate}
\item Suppose that $t^2+at+b$ is irreducible in $Q(R/\p)[t]$ and let $\q$ be the prime ideal of $\RR$ lying over $\p$. Then, $(\RR)_{\mathfrak q} \cong R_{\p}(I_{\p})_{a,b}$.
\item Suppose that $t^2+at+b=(t-\bar\alpha / \bar\gamma)(t- \bar\beta / \bar\gamma)$ in $Q(R/\p)[t]$ and let $\p_1, \p_2$ be the prime ideals of $\RR$ lying over $\p$.
\begin{enumerate}
    \item If $(\alpha-\beta)I \subseteq \p$ (i.e. $\p_1=\p_2$), then $(\RR)_{\p_i} \cong R_\p(I_\p)_{a,b}$ for $i=1,2$.
    \item If $(\alpha-\beta)I \nsubseteq \p$, i.e. there exists $\lambda \in I$ such that $(\alpha-\beta)\lambda \notin \p$, then $(\RR)_{\p_i} \cong R_\p$ for $i=1,2$, provided that there exists a choice of $\alpha$, $\beta$ and $\gamma$, such that $a\gamma=-\alpha-\beta$ and $b\gamma^2=\alpha \beta+p$ in $R$, where $p \in \p$ and $p\lambda I=0$. In particular, the last hypothesis holds if $t^2+at+b$ is reducible also in $Q(R)[t]$.
\end{enumerate}
\end{enumerate}
\end{prop}

\begin{proof}
(1) We have $s+jt \in \RR \setminus \mathfrak{q}$ if and only if at least one between $s$ and $j$ is in $R \setminus \p$.
Given an element $(r+it)/(s+jt) \in (\RR)_{\mathfrak q}$, we can multiply it by $(s-aj-jt)/(s-aj-jt)$; in fact, $s-aj-jt \in \RR \setminus \mathfrak{q}$,
if $j \in R \setminus \p$, but this also happens if $j \in \p$ and $s \in R \setminus \p$,
because in this case $s-aj \in R \setminus \p$.
Hence, we get an injective homomorphism between
$(\RR)_{\mathfrak q}$ and $R_{\p}(I_{\p})_{a,b}$ given by
$$
\frac{r+it}{s+jt} \longmapsto \frac{r+it}{s+jt} \cdot \frac{(s-aj-jt)}{(s-aj-jt)} = \frac{rs-ajr+ijb}{s^2-ajs+bj^2}+\frac{si-rj}{s^2-ajs+bj^2}t.
$$
\indent Moreover, it is surjective because a generic element $r/s + (i/s')t$ comes from
$(rs'+ist)/(ss') \in (\RR)_{\mathfrak{q}}$. \\
(2) (a) We recall that in this case $\p_1=\p_2$. If $s+jt \in \RR \setminus \p_1$, then $ja-s+jt \notin \p_1$, because $\gamma ja-\gamma s + \alpha j=-j\alpha -j \beta - \gamma s + \alpha j \notin \p$, since $s+jt \notin \p_2$. Therefore, given an element $(r+it)/(s+jt) \in (\RR)_{\p_1}$, one has
$$
\frac{r+it}{s+jt} \cdot \frac{ja-s+jt}{ja-s+jt} =
\frac{rja-rs-bij}{sja-s^2-bj^2} + \frac{rj-si}{sja-s^2-bj^2}t.
$$
Clearly, $sja-s^2-bj^2 \in R \setminus \p$, because $\p_1 \cap R= \p$, therefore, we get a well-defined ring homomorphism
$$
f: (\RR)_{\p_1} \rightarrow R_\p(I_\p)_{a,b}
$$
$$
f \left( \frac{r+it}{s+jt} \right)=
\frac{rja - rs -bij}{sja-s^2-bj^2} + \frac{rj-si}{sja-s^2-bj^2}t
$$
that is injective by construction.
As for the surjectivity, if $\frac{r}{s_1}+\frac{i}{s_2}t$ is an element of $R_\p(I_\p)_{a,b}$,
it is easy to see that this is equal to $f(\frac{rs_2+is_1t}{s_1 s_2})$. Hence, $f$ is an isomorphism and, therefore, $(\RR)_{\p_i} \cong R_{\p}(I_{\p})_{a,b}$ for $i=1,2$. \\
(2) (b) Consider the map $g_1 : R_\p \rightarrow (\RR)_{\p_1}, \ g_1 \left( \frac{r}{s} \right) = \frac{r}{s}$.
Clearly, this is well defined and is an injective ring homomorphism. As for the surjectivity consider a generic $\frac{r+it}{s+jt} \in (\RR)_{\p_1}$ and let $\lambda$ be an element of $I$ such that $\lambda (\alpha - \beta) \notin \p$. Then, $- \beta \lambda \gamma + \gamma^2 \lambda t \notin \p_1$ and it is easy to see that
$(r+it)(-\beta \lambda \gamma + \gamma^2 \lambda t)=-r\beta \lambda \gamma - \alpha \beta i \lambda - pi\lambda+(r \gamma^2 \lambda-i \beta \lambda \gamma + \alpha i \gamma \lambda + \beta i \gamma \lambda)t=(r \gamma + i \alpha)(- \beta \lambda + \gamma \lambda t)$, since $pi \lambda=0$. It follows that
$$
\frac{r+it}{s+jt} \cdot \frac{-\beta \lambda \gamma + \gamma^2 \lambda t}{-\beta \lambda \gamma + \gamma^2 \lambda t} =
\frac{r \gamma +i \alpha}{s \gamma + j \alpha}.
$$
Hence, $g_1 \left(\frac{r \gamma+i \alpha}{s \gamma + j \alpha} \right) = \frac{r+it}{s+jt}$ and $(\RR)_{\p_1} \cong R_{\p}$. The same argument can be applied to $(\RR)_{\p_2}$.
Finally, note that if $t^2+at+b$ is reducible in $Q(R)[t]$, then $p=0$.
\end{proof}

\begin{question}
With the notation of the previous proposition, if $(\alpha-\beta)I \nsubseteq \p$ and $p\lambda I \neq 0$, for any possible choice of $p$ and $\lambda$, is it still true that $(R(I)_{a,b})_{\p_i}$ is isomorphic to $R_\p$? 
\end{question}

We recall that an ideal is said to be regular if it contains a regular element.
In the light of the previous proposition, \cite[Proposition 2.7]{BDS} and \cite[Corollary 1.4]{BDS2}, we can immediately state the following corollary:

\begin{cor}
Let $R$ be a ring, let $I$ be an ideal of $R$ and let $a,b \in R$. Assume that $t^2+at+b$ is reducible in $Q(R)[t]$. Denote by $\mathbb{M}$ the set of all the maximal ideals $\m$ of $R$ except those for which $t^2+at+b=(t-\alpha_{\mathfrak m}/\gamma_{\mathfrak m})(t-\beta_{\mathfrak m}/\gamma_{\mathfrak m})$ in $(R/\m)[t]$ and $(\alpha_{\mathfrak m}-\beta_{\mathfrak m})I \nsubseteq \m$. Then:
\begin{enumerate}
\item The ring $\RR$ is Cohen-Macaulay if and only if $R$ is Cohen-Macaulay and $I_{\m}$ is a maximal Cohen-Macaulay $R_{\m}$-module for all $\m \in \mathbb{M}$;
\item Assume that $I_\m$ is regular for all $\m \in \mathbb{M}$. The ring $\RR$ is Gorenstein if and only if $R$ is Cohen-Macaulay, $I_{\m}$ is a canonical ideal of $R_{\m}$ for all $\m \in \mathbb{M}$, and $R_{\m}$ is Gorenstein for all $\m \notin \mathbb{M}$.
\end{enumerate}
\end{cor}

\begin{ex} \rm Let $k$ be an algebraically closed field of characteristic different from $2$ and assume that $R \cong k[x_1,\dots,x_n]/J$ is a domain.
Let $b \in R$ such that $t^2-b$ is irreducible in $Q(R)[t]$ (it is proved in \cite[Corollary 2.6]{BDS} that we can always find infinitely many such $b$), let $I$ be an ideal of $R$ and consider the ring $R(I)_{0,-b}$.
We can present it as a quotient of $k[x_1,\dots,x_n,y_1,\dots,y_m]$, where $m$ is the cardinality of a minimal set of generators of $I$; Corollary \ref{domain} implies that this is an integral domain. For any maximal ideal $\m_Q$ (corresponding to the point $Q$ in the affine variety $\mathcal V(J)$) the polynomial $t^2-b$ is reducible in $(R/\m_Q)[t]\cong k[t]$, since $k$ is algebraically closed; moreover, if $b\notin \m_Q$ and
if $\alpha \in R$ is such that $t^2-b=(t-\bar\alpha)(t+\bar{\alpha})$, 
we have that $\alpha \notin \m_Q$, i.e. the condition $2\alpha I \subset \m_Q$ is equivalent to 
$I \subset \m_Q$. Hence, $R(I)_{0,-b}$ is the coordinate ring of an affine variety double covering $\mathcal V(J)$, with ramification points $Q$ corresponding to the roots of $b$ and to those points $Q$, such that the corresponding ideal $\m_Q$ contains $I$.
By Proposition \ref{localizations}, the points of the double covering lying over a 
ramification point are all singular, since, by \cite[Remark 2.4]{BDS}, a local ring of the form
$\RR$ is regular if and only if $R$ is regular and $I=0$.
\end{ex}

\section{Serre's conditions}

A noetherian ring $R$ satisfies Serre's condition $(S_n)$ if
$\depth R_{\p} \geq \min\{n, \dim R_{\p}\}$ for any $\p \in \spec R$.
In the next proposition we study Serre's condition $(S_n)$ for $R(I)_{a,b}$, generalizing the particular case of the amalgamated duplication studied in \cite{BSTY}.
In this section we assume that $t^2+at+b$ is reducible in $Q(R)[t]$ and, if it is also reducible in $Q(R/\p)[t]$, we write $t^2+at+b=(t-\bar\alpha_\p/\bar\gamma_\p)(t-\bar\beta_\p/\bar\gamma_\p)$. We recall that $R(I)_{a,b}$ is noetherian if and only if $R$ is noetherian (cf. \cite[Proposition 1.11]{BDS}).

\begin{prop} \label{Sn}
Let $R$ be a noetherian ring. Then, $\RR$ satisfies Serre's condition $(S_n)$ if and only if $R$ satisfies Serre's condition $(S_n)$ and
$\depth I_\p \geq \min \{n, \dim R_\p \}$ for all $\p \in \spec R$ such that either $t^2+at+b$ is irreducible in $Q(R/\p)[t]$ or $(\alpha_\p - \beta_\p) I \subseteq \p$.
\end{prop}

\begin{proof}
Let $\mathfrak{q}$ be a prime ideal of $\RR$ and $\p = \mathfrak{q} \cap R$.
We distinguish two cases according to Proposition \ref{localizations}.
In both cases we notice that $\dim R_{\p}=\dim R_{\p}(I_{\p})_{a,b} = \dim (\RR)_{\mathfrak{q}}$.

\begin{itemize}
\item If $(\alpha_\p - \beta_\p)I \nsubseteq \p$, we have $\depth (\RR)_{\mathfrak q} = \depth R_p$; then, $\depth (\RR)_{\mathfrak q} \geq \min \{n, \dim (\RR)_{\mathfrak{q}} \}$ if and only if $\depth R_p    \geq \min \{n, \dim R_{\p}\}$.

\item If either $(\alpha_\p - \beta_\p)I \subseteq \p$ or $t^2+at+b$ is irreducible in $Q(R/\p)[t]$, it follows that $\depth (\RR)_{\mathfrak{q}}=\depth R_{\p}(I_{\p})_{a,b} = \min\{\depth R_{\p}, \depth I_{\p}\}$. Consequently, we get $\depth (\RR)_{\mathfrak{q}} \geq \min \{n, \dim (\RR)_{\mathfrak{q}}\}$
    if and only if  $\min\{\depth R_\p,\depth I_{\p}\} \geq \min \{n, \dim R_{\p} \}$. \qedhere
\end{itemize} 
\end{proof}

A noetherian ring $R$ satisfies the condition $(R_n)$ if $R_{\p}$ is regular for all $\p \in \spec R$ with $\height \p \leq n$.
Bagheri, Salimi, Tavasoli, and Yassemi ask when the condition $(R_n)$ holds for the amalgamated duplication, see \cite[Remark 3.9]{BSTY}. The next result gives the answer for the more general construction $\RR$.

\begin{prop} \label{Rn}
Let $R$ be a noetherian ring. Then, $\RR$ satisfies $(R_n)$ if and only if $R$ satisfies $(R_n)$ and $I_{\p}=0$ for all $\p \in \spec R$ with $\height \p \leq n$ and such that either $t^2+at+b$ is irreducible in $Q(R/\p)[t]$ or $(\alpha_\p - \beta_\p) I \subseteq \p$.
\end{prop}

\begin{proof}
Let $\mathfrak{q}$ be a prime ideal of $\RR$ such that
$\height \p \leq n$ and $\p= \mathfrak{q} \cap R$.
As in the previous proposition there are two cases:
\begin{itemize}
\item If $(\alpha_\p - \beta_\p)I \nsubseteq \p$, then $(\RR)_{\mathfrak{q}}$ is regular if and only if $R_{\p}$ is regular.
\item If either $(\alpha_\p - \beta_\p)I \subseteq \p$  or $t^2+at+b$ is irreducible in $Q(R/\p)[t]$, then $(\RR)_{\mathfrak{q}}$ is regular if and only if $R_{\p}(I_{\p})_{a,b}$ is regular and this is equivalent to $R_{\p}$ regular and $I_{\p}=0$, by \cite[Remark 2.4]{BDS}.
\end{itemize}
Putting together the two cases we get the thesis.
\end{proof}

If the polynomial $t^2+at+b$ is reducible in $R[t]$, as in the cases of idealization and amalgamated duplication, we can be more precise.

\begin{cor}
Let $I$ be a regular ideal of a noetherian ring $R$ and suppose that $t^2+at+b=(t-\alpha)(t-\beta)$ in $R[t]$. Then, $\RR$ satisfies $(R_n)$ if and only if $R$ satisfies $(R_n)$ and $n < \height (\alpha - \beta)I$.
\end{cor}

\begin{proof}
If $\RR$ satisfies $(R_n)$ it follows from the previous proposition that $R$ satisfies $(R_n)$. Suppose by contradiction that $n \geq \height (\alpha - \beta)I$ and let $\p$ be a minimal prime of $(\alpha - \beta)I$ such that $\height(\p)=\height((\alpha-\beta)I)$. By the previous proposition we have $I_{\p}=0$, but, if $x \in I$ is regular, this implies that $xs=0$ for some $s \in R \setminus \p$, a contradiction.

Conversely, we have $(\alpha - \beta)I \nsubseteq \p$ for any prime ideal $\p$ of $R$ of height less than or equal to $n$; hence the thesis follows from the previous proposition.
\end{proof}

The previous corollary implies that, if $I$ is regular and $\height (\alpha - \beta)I \leq n$, the rings $\RR$ never satisfies condition $(R_n)$. This is the case of idealization, since $\alpha = \beta=0$. As for amalgamated duplication the factorization is $t(t-1)$, hence $R(I)_{-1,0}$ satisfies the property $(R_n)$ if and only if $R$ satisfies ${\rm R}_n$ and $n<\height(I)$.

\section{The case $t^2+at+b$ reducible in $R[t]$}

In this section we always assume that $t^2+at+b=(t - \alpha)(t- \beta)$, where
$\alpha$ and $\beta$ are elements of $R$. Particular cases of this setting are 
both idealization and duplication, since the corresponding polynomials are $t^2$ and, respectively, $t(t-1)$.
Thus, we get a subfamily of the family of rings $\RR$; the interest in studying this subfamily comes from the facts that it is large enough (as we will see, we can obtain elements that are not isomorphic neither to an idealizazion nor to a duplication) and, for any ring $T$ in it, $R$ is naturally a $T$ module (cf. Remark \ref{hom}).

We recall that, if $R$ is reduced, the amalgamated duplication is always reduced, while the idealization is never reduced; in particular, in these two constructions this property doesn't depend on the ideal $I$.
Despite this, in the next example we show that could be some choices of $a$ and $b$ for which $R(I)_{a,b}$ can be reduced or not depending on the ideal $I$.

\begin{ex} \label{esempio 1} \rm
Let $k$ be a field of characteristic two and set
$R:=k[X,Y]/(XY)$, that is a reduced ring.
Denote the images of $X,Y$ in $R$ by $x,y$ and
consider $R(I)_{x,y^2}=R[It]/I^2(t^2+xt+y^2)$.
Notice that $(t^2+xt+y^2)=(t+(y+x))(t+y)$, since ${\rm char} \ \! R=2$.
If $I=(y)$, then $(y)^2(t^2+xt+y^2)=(y^2)(t^2+y^2)$ in $R[t]$, so $R(I)_{x,y^2} \cong R \ltimes I$
by \cite[Proposition 1.5]{BDS} and, in particular, it is not a reduced ring. \\
On the other hand, if $I=(x)$, then $(x)^2(t^2+xt+y^2)=(x^2)(t^2+xt)$ in $R[t]$.
If $r+it$ is a nilpotent element of $R(I)_{x,y^2}$, it follows that
$0=(r+it)^n=r^n+t( \dots )$ and thus $r=0$, since $R$ is reduced. Therefore, if $i=\lambda x$ for some $\lambda \in R$, we get
$0=(it)^n=(\lambda x)^n t^n=\lambda ^n x^{2n-1} t$ and this implies $Y|\lambda$ in $k[X,Y]$, that is
$i=\lambda_1 xy=0$ in $R$. This proves that $R(I)_{x,y^2}$ is reduced.
\end{ex}

In Corollary \ref{idealization and duplication} we will see that the last ring of the previous example is an amalgamated duplication. However, in Example \ref{esempio 2} we will produce
a ring of our subfamily that is not isomorphic neither to an idealization nor to an amalgamated duplication, proving that there are also new rings in the family we are studying.

\begin{rem} \label{hom} \rm
We note that in our case there exists a ring homomorphism
$$
\frac{R[t]}{((t- \alpha)(t- \beta))} \longrightarrow \frac{R[t]}{(t-\alpha)} \cong R,
$$
If we restrict to $R(I)_{a,b}$, we get a ring homomorphism $R(I)_{a,b} \rightarrow R$, that maps $s+jt$ to $s+j\alpha$; 
since there also exists a natural homomorphism $R \rightarrow R(I)_{a,b}$, 
any $R$-module $M$ is an $R(I)_{a,b}$-module and vice versa; moreover 
$\lambda_{R}(M)=\lambda_{R(I)_{a,b}}(M)$, where $\lambda$ denotes the length of a module.

In particular,  $R$ is an $\RR$-module with the scalar multiplication $(s+jt)r=sr+j\alpha r$, where $s+jt \in \RR$ and $r \in R$.
\end{rem}

We denote the minimal primes of $R$ by ${\rm Min}(R)$ and their intersection, the nilradical of $R$, by $N(R)$. Moreover, we write $\Ann(x)$ and $\Ann(I)$ to denote the annihilator of the element $x$ and the ideal $I$ respectively. 

\begin{prop} \label{reduced}
$\RR$ is reduced if and only if $R$ is reduced and $I \cap \Ann(\alpha - \beta)=(0)$.
\end{prop}

\begin{proof}
The set of minimal primes of $\RR$ is $A=\bigcup_{\p \in {\rm Min}(R)} \{ \p_1,\p_2 \}$,
therefore, $\RR$ is reduced if and only if $N(\RR)=\bigcap A = (0)$.

Assume that $R$ is reduced and $I \cap \Ann(\alpha - \beta)=(0)$. Let $r+it$ be an element of $N(\RR)$ and
fix $\p \in {\rm Min}(R)$. Since $r+it \in \p_1 \cap \p_2$, then $r+\alpha i$, $r+\beta i \in \p$ and consequently $i(\alpha - \beta) \in \p$. This holds for any $\p \in {\rm Min}(R)$ and, thus,
$i(\alpha - \beta) \in \bigcap\limits_{\p \in {\rm Min}(R)} \p = (0)$. This implies that $i=0$,
since $I \cap \Ann(\alpha - \beta)=(0)$,
and then also $r=0$, since $R$ is reduced.  

Conversely it is clear that $R$ is reduced if $\RR$ is.
Moreover, if $i \in I \cap \Ann(\alpha - \beta)$, then
$$
(- \beta i + it)^2 = \beta ^2 i^2 - b i^2 +(-2 \beta i^2 - ai^2)t =
(\beta - \alpha) i^2 \beta + (\alpha - \beta) i^2 t = 0,
$$
hence, $-\beta i +it=0$ and consequently $i=0$.
\end{proof}

\begin{cor}
Let $R$ be a reduced ring and assume that $\alpha - \beta$ is regular, then $\RR$ is reduced. Moreover, the converse holds if $I$ is regular.
In particular, if $R$ is an integral domain and $\alpha \neq \beta$, then $\RR$ is reduced.
\end{cor}

\begin{proof}
The first part follows immediately from the previous proposition. Conversely, if by contradiction there exists $x \in R$ such that $x(\alpha-\beta)=0$ and $i \in I$ is regular, then $0 \neq xi \in I \cap \Ann(\alpha-\beta)$ and the previous proposition leads to a contradiction.
\end{proof}

\begin{rem} \rm
We note that, if $t^2+at+b$ is irreducible in $Q(R/\p)[t]$ for any $\p \in {\rm Min}(R)$, then $R$ is reduced if and only if $\RR$ is reduced. In fact, if $R$ is reduced it is enough to compute the nilradical of $\RR$:
$$
N(\RR)= \bigcap\limits_{\mathfrak q \in {\rm Min}(\RR)} \mathfrak q
= \bigcap\limits_{\p \in {\rm Min}(R)} \{p+it \ | \ p \in \p, i \in I \cap \p \}=
$$
$$
= \{p+it \ | \ p \in N(R), i \in I \cap N(R) \} = (0)\RR.
$$
\end{rem}

\subsection{Idealization and amalgamated duplication}

We have already noted that the idealization $R \ltimes I$ and the amalgamated duplication $\RI$ are members of our family;
in this subsection we study when $R(I)_{a,b}$ is isomorphic to them.
As consequence, we will show that it is possible to find rings in our subfamily that are not isomorphic neither to an idealization nor to an amalgamated duplication (cf. Example \ref{esempio 2}).

\begin{prop}
The following statements hold. \\
$1)$ $\RR \cong R \ltimes I$ if $\alpha - \beta \in \Ann(I^2)$. \\
$2)$ $\RR \cong R \Join (\alpha-\beta)I$ if $\Ann(\alpha - \beta) \cap I=(0)$.
\end{prop}

\begin{proof}
We can consider the ring automorphism of $R[t]$ given by $t \mapsto t+ \alpha$, then
$$
\RR = \frac{R[It]}{I^2((t-\alpha)(t-\beta))} \cong
\frac{R[It]}{I^2(t^2+(\alpha - \beta)t)} = R(I)_{\alpha - \beta,0}.
$$
$1)$ If $\alpha - \beta \in \Ann(I^2)$, then
$$
\RR \cong \frac{R[It]}{I^2(t^2+(\alpha -\beta)t)}
=\frac{R[It]}{I^2(t^2)} \cong R \ltimes I
$$
by \cite[Proposition 1.4]{BDS}. \\
$2)$ Consider the map
$\varphi: R(I)_{\alpha-\beta,0} \rightarrow R((\alpha-\beta)I)_{-1,0}$
given by $\varphi (r+it)=r + (\alpha - \beta)it$. This is a ring homomorphism, since
$$
\varphi((r+it)(s+jt))\! = \! \varphi(rs+(rj+si+ij(\alpha - \beta))t) \! = \! rs + (\alpha - \beta)(rj + si +ij(\alpha - \beta))t
$$
$$
\varphi(r+it)\varphi(s+jt) \! = \! (r+(\alpha - \beta)it)(s + (\alpha - \beta)jt) \! = \!
rs + (\alpha - \beta)(rj + si +ij(\alpha - \beta))t.
$$
Moreover, $\varphi$ is clearly surjective and, since
$\Ann(\alpha - \beta) \cap I=(0)$, it is injective as well;
hence, $\varphi$ is an isomorphism and the thesis follows, because
$R((\alpha-\beta)I)_{-1,0} \cong R \Join (\alpha-\beta)I$ by \cite[Proposition 1.4]{BDS}.

\end{proof}

\begin{cor} \label{idealization and duplication}
Let $R$ be a reduced ring.
The following statements hold. \\
$1)$ $\RR \cong R \ltimes I$ if and only if $\alpha - \beta \in \Ann(I)$. \\
$2)$ $\RR \cong R \Join (\alpha - \beta)I$ if and only if $\Ann(\alpha - \beta) \cap I=(0)$.
\end{cor}

\begin{proof}
$1)$ If $\alpha - \beta \in \Ann(I)$, then $\RR \cong R \ltimes I$ by previous proposition.
Conversely, suppose that $\RR \cong R \ltimes I$.
Then, only one prime ideal of $\RR$ lies over a prime ideal of $R$, and this happens if and only if
$(\alpha - \beta)I \subseteq \bigcap\limits_{\p {\rm \ prime}} \p = (0)$, because $R$ is reduced; hence $\alpha - \beta \in \Ann(I)$. \\
$2)$ We need to prove that, if $\RR \cong R \Join (\alpha - \beta)I$, then $\Ann(\alpha - \beta) \cap I=(0)$. If this does not happen, $\RR$ is not reduced by Proposition \ref{reduced} and this yields a contradiction, since the amalgamated duplication is reduced if $R$ is.
\end{proof}

The next example shows that in the subfamily studied in this section there are rings that are not isomorphic neither to an idealization nor to an amalgamated duplication.

\begin{ex} \label{esempio 2} \rm
Consider $R=k[X,Y]/(XY)$ and $I=(x,y)$, where $k$ is a field with ${\rm char} \ \! k \neq 2$ and $x,y$ denote the images of $X,Y$ in $R$.
Set $\alpha = y-r$ and $\beta = -y -r$ with $r \in R$, thus,
$(t-\alpha)(t-\beta)=t^2+2rt+r^2-y^2$.
We have $\alpha - \beta=2y \notin \Ann(I)$ and
$x \in \Ann(\alpha - \beta) \cap I$, then $R(I)_{2r,r^2-y^2}$ has two different minimal prime ideals by Lemma \ref{ideals} and Proposition \ref{minimal primes}; consequently, it cannot be isomorphic to an idealization.
Moreover, since $\Ann(\alpha - \beta) \cap I \neq (0)$, Proposition \ref{reduced} implies that $R(I)_{2r,r^2-y^2}$ is not reduced and, therefore, it is not isomorphic to an amalgamated duplication.
\end{ex}

If $R$ is not reduced, the first part of Corollary \ref{idealization and duplication} does not hold, as it is shown in the next example.

\begin{ex} \rm
Consider $R=\Z/2^n \Z$. The non units of $R$ are the classes represented by $2^\alpha m$ with $\alpha \geq 1$ and $m$ odd. Then, the square of any ideal is annihilated by $2^{n-2}$. This means that for any ideal $I$ one has
$R[It]/I^2 (t^2+2^{n-2}t) = R[It]/I^2 (t^2) \cong R \ltimes I$.
Moreover, if we choose $I=(2)$ we have $2^{n-2} \notin \Ann(I)$.
\end{ex}

\subsection{Quasi-Gorenstein rings}

Let $(R, \m)$ be a $d$-dimensional local ring. A finitely generated $R$-module $\omega_R$ is said to be a canonical module of $R$ if $\omega_R \otimes_R \widehat{R} \cong \Hom_R(H^d_{\m}(R), E(R/\m))$, where $H^d_{\m}(R)$ denotes the $d$-th local cohomology module of $R$ with support in $\m$ and $E(R/\m)$ is the injective hull of $R/\m$.
If the canonical module $\omega_R$ exists, it is unique up to isomorphism. In this case the ring $R$ is said to be quasi-Gorenstein if its canonical module is a rank one free $R$-module, see \cite{TT} and references therein for other characterizations and properties of quasi-Gorenstein rings. 
In \cite{A2}, Aoyama characterizes when idealization is quasi-Gorenstein, while Bagheri, Salimi, Tavasoli and Yassemi do the same for amalgamated duplication in \cite{BSTY}. In this subsection we generalize these results to all the rings of the family $R(I)_{a,b}$ for which $t^2+at+b$ is reducible in $R[t]$.
We start by recalling a lemma that we will use in Theorem \ref{quasi Gorenstein}.

\begin{lem} \label{Completion} {\cite[Remark 1.7]{BDS2}}
Let $R$ be a noetherian local ring.
Then $\widehat{\RR} \cong \widehat{R}(\widehat{I})_{a,b}$ as $\widehat{R}$-algebra.
\end{lem}

In this section we consider $R$ as an $\RR$-module with scalar
multiplication $(s+jt)r= rs+rj \alpha $ as in Remark \ref{hom}.

\begin{lem}
If $\Ann(I)=(0)$, then
$I \cong {\rm Hom}_{\RR}(R,\RR)$ as $R$-modules.
\end{lem}

\begin{proof}
For any $r \in R$ and $i \in I$ we set $g_i : R \rightarrow \RR$ with $g_i(r)=ri \beta - rit$,
that is an homomorphism of $\RR$-modules. Consider the map
$$
\varphi : I \rightarrow {\rm Hom}_{\RR}(R,\RR), \ i \mapsto g_i;
$$
it is easy to prove that this is an injective homomorphism of $R$-modules.

We claim that $\varphi$ is also surjective.
Consider $h \in {\rm Hom}_{\RR}(R,\RR)$, clearly this is determined by
$h(1) = s + jt$ where $s \in R$ and $j \in I$.
Since $h$ is an homomorphism of $\RR$-modules, for any $i \in I$ we have
$$
rs+rjt = r h(1)=h(r) = h(r-i \alpha + i \alpha) =
h((r-i \alpha + it) \cdot 1) =
$$
$$
=(r-i \alpha + it)h(1)=
(r-i \alpha + it)(s+jt)=
$$
$$
 = rs - i\alpha s- ij \alpha \beta +
(rj - ij \alpha + si + ij \alpha + ij \beta) t,
$$
then, $h$ is well defined only if for any $i \in I$

$$
\begin{cases}
i(-s \alpha -j \alpha \beta)=0 \\
i(s+j \beta)=0
\end{cases}
$$
and this implies that $s=-j \beta$ because $\Ann(I)=(0)$.
Hence $h=g_{-j}$ and $\varphi$ is an isomorphism.
\end{proof}

The following result is a generalization of \cite[Theorem 2.11]{A2} and \cite[Theorem 3.3]{BSTY}. The idea of the first implication is essentially the same of \cite{A2} and \cite{BSTY}, but it requires the previous two lemmas and Proposition \ref{Sn}. We recall that an ideal of $R$ is said to be a canonical ideal if it is a canonical module of $R$.

\begin{thm} \label{quasi Gorenstein}
Let $R$ be a noetherian local ring and suppose that $I$ is regular. Then, $\RR$ is quasi-Gorenstein if and only if $\widehat{R}$ satisfies Serre's condition $(S_2)$ and $I$ is a canonical ideal of $R$.
\end{thm}

\begin{proof}
If $\RR$ is quasi-Gorenstein, it is well known that also $\widehat{R(I)_{a,b}} \cong \widehat{R}(\widehat{I})_{a,b}$ is quasi-Gorenstein. Consequently, since a canonical module always satisfies the condition $(S_2)$ (see \cite[Theorem 12.1.18]{BS}), it follows that $\widehat{R}(\widehat{I})_{a,b}$ satifies $(S_2)$ and, in the light of Proposition \ref{Sn}, also $\widehat{R}$ satisfies $(S_2)$. Moreover, since we have an homomorphism $\RR \rightarrow R$ (see Remark \ref{hom}), it follows from \cite[Satz 5.12]{HK} (or \cite[Theorem 1.2]{A2}) and the previous lemma that $\Hom_{\RR}(R,\RR) \cong I$ is a canonical module.

Conversely, using again \cite[Satz 5.12]{HK}, we get that a canonical modules of $\RR$ is $\Hom_R(\RR,I)$, because $I$ is a canonical ideal of $R$. To prove that $\RR$ is quasi-Gorenstein we only need to show an isomorphism $\varphi: \RR \rightarrow \Hom_R(\RR,I)$. To this aim, we set $f_{r+it}(s+jt)=rj+i(s-ja) \in I$ and $\varphi(r+it)=f_{r+it}$. To check that this is an isomorphism, it is possible to follow the same proof of \cite[Proposition 2.1]{BDS2}, bearing in mind that for the surjectivity one can use that $(I:I) \hookrightarrow \Hom_R(I,I) \cong R$, because $\widehat{R}$ satisfies Serre's condition $(S_2)$, see \cite[Proposition 2]{A}.
\end{proof}

\begin{rem} \rm
We notice that by the proof above, if $R$ satisfies the Serre's condition $(S_2)$ and $I$ is a canonical ideal, then $\RR$ is quasi-Gorenstein even if $t^2+at+b$ is not reducible in $R[t]$.
\end{rem}

\end{document}